\documentclass{amsart}
\usepackage[pagewise]{lineno}
\newcommand{\R}{{\mathbb{R}}}

\newtheorem{definition}{Definition}
\newtheorem{theorem}{Theorem}

\newtheorem{lemma}{Lemma}

\newtheorem{proposition}{Proposition}
\newtheorem{remark}{Remark}
\newtheorem{example}{Example}

\begin{document}
%\linenumbers
\title[]{Nonuniform Almost Reducibility of Nonautonomous Linear Differential Equations }
\author[\'A. Casta\~neda]{\'Alvaro Casta\~neda}
\author[I. Huerta]{Ignacio Huerta}
\address[\'A. Casta\~neda]{Departamento de Matem\'aticas, Facultad de Ciencias, Universidad de
 Chile, Casilla 653, Santiago, Chile}
\address[I. Huerta]{Departamento de Matem\'atica y C.C, Facultad de Ciencia, Universidad de Santiago de Chile
Casilla 307, Correo 2, Santiago, Chile}
\email{castaneda@uchile.cl, ignacio.huerta@usach.cl}
\thanks{This research has been partially supported  FONDECYT Regular 1170968. The second author is also supported by (CONICYT-PCHA/2015-21150270)}
\date{\today}

\maketitle
\begin{abstract}
We prove that a linear nonautonomous differential system with nonuniform hyperbolicity on the half line  can be written  as diagonal system with a perturbation which is small enough. Moreover we show that the diagonal terms are contained in the nonuniform exponential dichotomy spectrum. For this purpose we
introduce the concepts of \textit{nonuniform almost reducibility} and \textit{nonuniform contractibility}  which are generalization of these notions originally defined in a uniform context.
\end{abstract}

\section{Introduction}

Given a linear operator $T$, to find an ordered basis in which $T$ assumes an especially simple form is a classic problem in linear algebra. When we are working in finite dimensional spaces, this problem has
a strong relation with the study of the dynamics of a linear differential equation
\begin{equation}
\label{lin}
\dot{x} = A(t)x,
\end{equation}
where $x \in \mathbb{R}^n$ and $t\in J$ with $J\subseteq \R$ an interval.
\subsection{Autonomous and Nonautonomous contexts}
When $A(t) = A,$ the know\-led\-ge of the real part of eigenvalues of $A$ allows us to construct the stable and unstable invariant manifolds and the resulting canonical form of $A$ gives some insights into the form of the solutions of (\ref{lin}).

In a nonautonomous context, the problem of finding a simpler form of the matrix $A(t)$ and as a consequence to study the qualitative behavior of (\ref{lin}) is a more delicate task. In fact, contrarily to the autonomous case, the eigenvalues analysis does not always allow any conclusion over the stability of the solutions, and thus an alternative focusing must be considered.

The first approach in this direction was given by G. Floquet \cite{Floquet}, which established that a periodic system can be transformed into a constant coefficients system. Floquet's result can be seen
as an example of the properties of \emph{kinematical si\-milarity} and \emph{reducibility}, which refers that a linear system (\ref{lin}) can be transformed into
\begin{equation}
\label{reducido}
\dot{y}=B(t)y
\end{equation}
through a coordinate change $x = L(t)y$, where the invertible matrix function $L(t)$ is known as Lyapunov transformation.

The problem to obtain a simpler form of (\ref{lin}) has been tackled by using the concept of {\emph{reducibility} by O. Perron in \cite{Perron}, which proves that (\ref{lin}) can be reduced via unitary transformation
to a system (\ref{reducido}) where $B(t)$ has a triangular form whose diagonal coefficients are real. Moreover, under subtle technical conditions it can be proved that $B(t)$ has a block--triangular form consisting
of blocks whose diagonal coefficients are real.

We have mentioned that an eigenvalues--based approach has several shortcomings and is not an adequate tool to cope with stability issues in the nonautonomous framework. A tool that emulates the role of the eigenvalues in this context was developed in terms of the property of {\emph{uniform exponential dichotomy}} (a type of nonautonomous hyperbolicity), namely, the Sacker--Sell spectrum associated to (\ref{lin}), which is the set
$$\sigma(A) =  \{\lambda \in \mathbb{R} \colon \dot{x} = (A(t) - \lambda I)x \, \, \textnormal{has no uniform exponential dichotomy on $J\subset \mathbb{R}$}  \}.$$

This spectrum plays a fundamental role in a better localization of diagonal terms when the system (\ref{lin}) can be transformed to a diagonal one. In fact, B. F. Bylov in \cite{Bylov} introduced the notion of {\emph{almost reducibility}}, \textit{i.e.}, reducibility with a negligible error and proved that any linear system is almost reducible to some diagonal system with real coefficients. Later, F. Lin in \cite{Lin} improves the Bylov's result by proving that the diagonal coefficients are contained in the Sacker--Sell Spectrum. Moreover, F. Lin proved that this spectrum is the minimal compact set where the diagonal terms belong to, this phenomenon is known as the {\emph{contractibility}} of a linear system.

We emphasize that these concepts of reducibility and almost reducibility also have a vast literature as well as in the uniform hyperbolicity (\cite{Coppel}, \cite{Johnson}) or in Schr\"{o}dinger operators
\cite{Yu}.

\subsection{Structure and novelty of the article} The section 2 introduces the concepts of \emph{nonuniformly almost reducible} and \emph{nonuniformly contractible} systems, both notions
are the generalizations of the ideas of almost reducibility and contractibility previously mentioned. Instead of using the set $\sigma(A)$, we use the spectrum of the \emph{nonuniform exponential dichotomy}
introduced in \cite{Chu, Zhang}. The main result of section 3 states that
if the linear system (\ref{lin}) verifies a subtle condition of nonuniform hyperbolicity on $J=\mathbb{R}_{0}^{+}$,  then this system is uniformly contracted to the spectrum of nonuniform exponential dichotomy (a formal definition will be given later).  The section 4 deals with preparatory Lemmas to obtain the main result, which is proved in the section 5. Finally, in section 6, we give an scalar and planar applications of our principal theorem.

Our main result is a generalization of the Lin's work in \cite{Lin}. In spite that our proof follows the lines of that article, it is worth to stressing that, compared with the uniform case, the nonuniform behavior of the solutions of (\ref{lin}) combined with our restriction to $J=\mathbb{R}_0^{+}$ arises technical subtleties and bulky technicalities (namely conditions \textnormal{\textbf{{(C1)--(C4)}}} in the section 5) in order to obtain the desired result, which deserve interest on itself.

\section{Preliminaries}
We consider the linear system (\ref{lin})
with $x$ as a column vector of $\R^{n}$ and the matrix function $t \mapsto A(t) \in \R^{n\times n}$ with the following properties:
\begin{itemize}
\item [\textbf{(P1)}] There exists a couple of positive numbers $\mu, \mathcal{M} > 0,$ such that $$\left \|A(t)\right \|\leq \mathcal{M} \exp(\mu t) $$  for any $t\in\R_0^{+}.$ 
\vspace{0.3 cm}
\item [\textbf{(P2)}] The evolution operator $\Phi (t,s)$ of (\ref{lin}) has a nonuniformly bounded growth (\cite{Zhang}); namely, there exist constants $K_0\geq1$, $a\geq0$ and $\bar{\varepsilon}\geq0$ such that $$\left \| \Phi (t,s) \right \| \leq K_0\exp(a\left | t-s \right | + \bar{\varepsilon}s), \quad t,\ s \in\R_0^{+},$$
\end{itemize}
where $\left \| \cdot\right \|$ denotes a matrix norm and

\begin{displaymath}
x(t) = \Phi(t,s) x(s), \quad \quad \Phi(t,s) \Phi(s, \tau) = \Phi(t, \tau), \quad \text{for \, \, all} \quad t,s,\tau \in \R_{0}^+.
\end{displaymath}

The purpose of this article is to study the nonuniform contractibility or nonuniform almost reducibility to a diagonal system.
Namely, the $\delta$-nonuniform kinema\-tical similarity of (\ref{lin}) to
\begin{equation}
\label{perturbado}
\dot{y}=U(t)y,
\end{equation}
where $U(t)=C(t)+B(t)$; $C(t)$ is a diagonal matrix and $B(t)$ has smallness properties which will be explained later.

\begin{definition}\textnormal{(\cite{Zhang})}
The system (\ref{lin}) is nonuniformly kinematically similar (resp. $\delta-$nonuniformly kinematically similar with a fixed $\delta>0$) to (\ref{perturbado}) if there exist an invertible transformation $S(t)$ (resp. $S_{\delta}(t)$) and $\upsilon\geq0$ satisfying
$$\left \|S(t)\right \|\leq M_{\upsilon}\exp(\upsilon t) \quad and \quad \left \|S^{-1}(t)\right \|\leq M_{\upsilon}\exp(\upsilon t)$$
or respectively
$$\left \|S(\delta, t)\right \|\leq M_{\upsilon, \delta}\exp(\upsilon t) \quad and \quad \left \|S^{-1}(\delta, t)\right \|\leq M_{\upsilon, \delta}\exp(\upsilon t),$$
such that the change of coordinates $y(t)=S^{-1}(t)x(t)$ (resp. $y(t)=S_{\delta}^{-1}(t)x(t)$) transforms (\ref{lin}) into (\ref{perturbado}), where

\begin{equation}
\label{scnu}
U(t)=S^{-1}(t)A(t)S(t)-S^{-1}(t)\dot{S}(t),
\end{equation}
for any $t \in\R_0^{+}$.
\end{definition}

\begin{remark}
The nonuniform kinematical similarity preserves nonuniformly growth bounded. In fact, if (\ref{lin}) and (\ref{perturbado}) are nonuniform kinematically similar through of the function $S(\cdot)$ and their respective evolution operators are $\Phi_1(t,s)$ and $\Phi_2(t,s)$, then by lemma 3.1 in \cite{Zhang} we have the invariance property
$$\Phi_1(t,s)S(s)=S(t)\Phi_2(t,s) \quad \text{for} \quad \text{all} \quad t, s \in\R_0^{+},$$
and if $\left \| \Phi_1 (t,s) \right \| \leq K_0\exp(a\left | t-s \right | + \bar{\varepsilon}\left | s \right |)$, then we verify that
$$
\begin{array}{rl}
\left \| \Phi_2 (t,s) \right \| &\leq \left \| S^{-1}(t) \right \| \left \| S(s) \right \| \left \| \Phi_1 (t,s) \right \|,\\
&\leq M_{\upsilon}\exp(\upsilon t)M_{\upsilon}\exp(\upsilon s)K_0\exp(a|t-s|+\bar{\varepsilon} s),
\end{array}
$$
and finally, we obtain that
$$\left \| \Phi_2 (t,s) \right \| \leq M_{\upsilon}^{2}K_0\exp((\upsilon+a)|t-s|+(2\upsilon+\bar{\varepsilon})s).$$
\end{remark}

As we said previously, the concept of almost reducibility was introduced in the 60's by B. F. Bylov in the continuous context. A discrete version of this notion was given by \'A. Casta\~neda and  G. Robledo (see \cite{CR2017}).

 Now we introduce the
definition of \textit{nonuniformly almost reducible}, which is a version of the previous concept in the nonuniform framework.

\begin{definition}
The system (\ref{lin}) is nonunifomly almost reducible to
$$\dot{y}=C(t)y,$$
if for any $\delta>0$ and $\varepsilon\geq0$, there exists a constant $K_{\delta,\varepsilon}\geq1$ such that (\ref{lin}) is $\delta-$nonuniformly kinematically similar to
$$\dot{y}=[C(t)+B(t)]y,\quad with \quad \left \|B(t)\right \|\leq \delta K_{\delta,\varepsilon}$$
for any $t\in\R_0^{+}$.
\end{definition}

\begin{remark}
\label{remarkpartenouniforme}

\noindent \begin{itemize}

\item[(i)] It is appropriate to say that the parameter $\varepsilon$ represents the nonuniform part in the definition of nonuniform exponential dichotomy of the linear system (\ref{lin}), as will be seen in the examples that will be presented at the end of this work.

\item [(ii)] In the case when $C(t)$ is a diagonal matrix, if $K_{\delta,\varepsilon}=1$, it is said that (\ref{lin}) is almost reducible to a diagonal system and it was proved by B.F. Bylov in \cite{Bylov} that any continuous linear system satisfies this property and the components of $C(t)$ are real numbers.
\end{itemize}
\end{remark}

The concept of almost reducibility to diagonal system was rediscovered and improved in the 90's by F. Lin in \cite{Lin}, who introduces the concept of contractibility in the continuous context, while in the discrete case was proposed by \'A. Casta\~neda and G. Robledo in \cite{CR2017}. In this paper we introduce its nonuniform version.

\begin{definition}
The system (\ref{lin}) is nonuniformly contracted to the compact subset $E\subset\R$ if is nonuniformly almost reducible to a diagonal system
$$\dot{y}=Diag(C_1(t),\dots, C_n(t))y,$$
where $C_i(t)\in E$, for any $t\in\R_0^{+}.$
\end{definition}

It is worth to emphasize that while Bylov's result only says that the diagonal components are real numbers, Lin's definition provides explicit localization properties,
as the fact that a compact set is contractible if it is the minimal compact set such that the system (\ref{lin}) can be contracted.

In the continuous and discrete cases, the concept of contractibility has been applied in some results of topological equivalence and almost topological equivalence respectively (see \cite{Lin2}, \cite{CR2018}). The major contribution of \cite{Lin} is to prove that the contractible set of a linear system (\ref{lin}) is its Sacker and Sell spectrum (see \cite{SS}). Mimicing the construction of the Sacker and Sell spectrum, S. Siegmund in \cite{Siegmund2002} define the nonuniform spectrum $\Sigma(A)$  (a formal definition will be given later). To the best of our knowledge, there exists no result in the nonuniform framework and the purpose of this article is to obtain condition for the nonuniform contractibility of (\ref{lin}) to $\Sigma(A)$ by following some lines of Lin's work in \cite{Lin}. 

\section{Main result: Nonuniform almost reducibility to diagonal systems and nonuniform spectrum.}

\subsection{Dichotomy and nonuniform spectrum.}

In this section we recall the concept of \textit{nonuniform exponential dichotomy} introduced by L. Barreira and C. Valls in \cite{BV-CMP} and its associated spectrum with some properties.

\begin{definition}\textnormal{(\cite{BV-CMP}, \cite{Chu}, \cite{Zhang})}
The system (\ref{lin}) has a nonuniform exponential dichotomy on $J\subset \R$ if there exist an invariant projector $P(\cdot)$, constants $K\geq 1$, $\alpha>0$ and $\varepsilon\geq 0$ such that
\begin{equation}
\begin{array}{rcl}
%\label{nouniforme}
\left \| \Phi(t,s)P(s) \right \|&\leq& K\exp(-\alpha(t-s)+\varepsilon\left | s \right |),\quad t\geq s, \quad t, s \in J,                      \\
\left \| \Phi(t,s)(I-P(s)) \right \|&\leq& K\exp(\alpha(t-s)+\varepsilon\left | s \right |),\quad t\leq s, \quad t, s \in J.
\end{array}
\end{equation}
\end{definition}

\begin{remark} We have the following comments with respect to this nonuniform dichotomy:
\begin{enumerate}
\item In the definition of nonuniform exponential dichotomy the condition  $\varepsilon < \alpha$  appears in \cite{Chu} and \cite{Zhang}, allowing them to show their respective technical results. Our definition is inspired in the work of Y. Xia \textit{et al.} \cite{Xia2019}.
\item It is considered a projector $P(t)$ that satisfies the equation $$P(t)\Phi(t,s)=\Phi(t,s)P(s)$$ and it is invariant in the following sense $$dim(Ker(P(t)))=dim(Ker(P(s))),$$ for all $t, s\in J$.
\end{enumerate}
\end{remark}

\begin{definition}\textnormal{(\cite{Chu} ,\cite{Zhang})}
The nonuniform spectrum (also called nonuniform exponential dichotomy spectrum) of (\ref{lin}) is the set $\Sigma(A)$ of $\lambda\in\R$ such that the systems
 \begin{equation}
 \label{sistemaperturbado}
\dot{x}=[A(t)-\lambda I]x
 \end{equation}
 have not nonuniform exponential dichotomy on $\R_0^{+}.$
\end{definition}

\begin{remark}
The evolution operator of (\ref{sistemaperturbado}) is $\Phi_{\lambda}(t,s)=\exp(-\lambda (t-s))\Phi(t,s)$. Moreover, if $\lambda\notin\Sigma(A)$, then there exist constants $K\geq1$, $\alpha>0$, $\varepsilon\geq 0$ and an invariant projector $P(t)$ such that $\Phi_{\lambda}(t,s)$ satisfies the following estimations for $t,s \in J$:
\begin{equation}
\begin{array}{rcl}
\label{nouniforme}
\left \| \exp(-\lambda(t-s))\Phi(t,s)P(s) \right \|&\leq& K\exp(-\alpha(t-s)+\varepsilon\left | s \right |),\quad t\geq s,                \\
\left \| \exp(\lambda(t-s))\Phi(t,s)(I-P(s)) \right \|&\leq& K\exp(\alpha(t-s)+\varepsilon\left | s \right |),\quad t\leq s.
\end{array}
\end{equation}
\end{remark}

\begin{remark}
If $\lambda \notin \Sigma(A),$ then $\lambda$ belongs to the resolvent set of $A,$ which is denoted by $\rho(A).$
\end{remark}

The following result allows us to give a better description of the spectrum when the evolution operator of system (\ref{lin}) has a nonuniformly bounded growth.

\begin{proposition}\textnormal{(\cite{Half-line}, \cite{Kloeden}, \cite{Siegmund2002}, \cite{Xia2019}, \cite{Zhang})} If the evolution operator of (\ref{lin}) satisfies \textnormal{\textbf{(P2)}}, then its nonuniform spectrum $\Sigma(A)$ is the union of m compact intervals where $0<m\leq n$, namely,
\begin{equation}
\label{espectronouniforme}
\Sigma(A)=\bigcup_{i=1}^{m}[ a_i,b_i ],
\end{equation}
with $-\infty<a_1\leq b_1<\ldots<a_m\leq b_m <+\infty$.
\end{proposition}

\begin{remark}
Notice that in \cite[Theorem 5]{Half-line} this result is done in the discrete framework with $J = \mathbb{N}.$ While that the rest of references are immersed in the continuous context. For more details see \cite[Theorem 5.12]{Kloeden}, \cite[Theorem 3.1]{Siegmund2002}, \cite[Corollary 1.8]{Xia2019} and \cite[Theorem 1.2]{Zhang}.
\end{remark}

The following result allows characterizing the nonuniformly bounded growth of the evolution operator associated to (\ref{lin}) from subtle hypothesis about its nonuniform spectrum.

\begin{proposition}
Suppose that  the system (\ref{lin}) has spectrum  $\Sigma(A)=[a,b]$, then its evolution operator $\Phi(t,s)$ satisfies \textnormal{\textbf{(P2)}}.
\end{proposition}
\begin{proof}
Let $\gamma, \lambda \in \rho(A)$ such that $\gamma < a\leq b<\lambda$, then we have that the system $$\dot{x}=(A(t)-\gamma I)x$$ has a nonuniform exponential dichotomy with projector $P(t)=0.$ On the other hand, the system $$\dot{x}=(A(t)-\lambda I)x$$ has a nonuniform exponential dichotomy with projector $P(t)=I.$

Therefore there exist $\alpha_1, \alpha_2, >0$, $\varepsilon_1, \varepsilon_2 \geq 0$, $K_1, K_2\geq1$ such that satisfies
$$\left \| \Phi(t,s) \right \|\leq K_1 \exp((\gamma+\alpha_1)(t-s)+\varepsilon_1 s) \quad (t\leq s),$$
$$\left \| \Phi(t,s) \right \|\leq K_2 \exp((\lambda-\alpha_2)(t-s)+\varepsilon_2 s) \quad (t\geq s).$$
Now we define $a=\max\left\{ 0,-\gamma-\alpha_1,\lambda-\alpha_2 \right\}$, $\varepsilon=\max\left\{ \varepsilon_1, \varepsilon_2\right\}$ and $K=\max\left\{K_1, K_2 \right\}$ which allow to conclude that
$$\left \| \Phi(t,s) \right \|\leq K \exp(a|t-s|+\varepsilon s) \quad (t,s\in \R_0^{+}).$$
\end{proof}

An important property of the spectrum is its decompose when the system is diagonal, as we can see in the following lemma. We point out that this result will be useful in the last section when we give planar examples of our main result.

\begin{lemma}
\label{lemaespectrodiagonal}
Consider the system \textnormal{(\ref{perturbado})} where $U$ is written as 
\begin{equation}
\label{Benbloques}
U(t)=\begin{pmatrix}
U_1(t) & 0\\
0 & U_2(t)
\end{pmatrix}
\end{equation}
where $U_1:\R_0^{+}\rightarrow M_{n_1}(\R), U_{2}:\R_0^{+}\rightarrow M_{n_2}(\R)$,
therefore the spectra satisfy 
\begin{equation}
\label{espectrodesistemaenbloques}
\Sigma(U)=\Sigma(U_1)\cup\Sigma(U_2).
\end{equation}
\end{lemma}
\proof{If $\lambda\in(\rho(U_1)\cap\rho(U_2))$, then the systems 
$$\dot{x}_1=(U_1(t)-\lambda I)x_1$$
and 
$$\dot{x}_2=(U_2(t)-\lambda I)x_2$$
with evolution operators $\Phi_{\lambda,1}(t,s)$ and $\Phi_{\lambda,2}(t,s)$ respectively, have a nonuniform exponential dichotomy, i.e., there exist constants $K_1,K_2\geq1, \alpha_i>0, \mu_i\geq0$,
%, with $\alpha_i>\mu_i$ 
for $i\in\left \{ 1,2\right \}$ and invariant projectors $P_1:\R_0^{+}\rightarrow M_{n_1}(\R), P_2:\R_0^{+}\rightarrow M_{n_2}(\R)$ respectively, satisfying 
\begin{equation*}
\left \{
\begin{array}{rcl}
%\label{nouniforme}
\left \| \Phi_{\lambda,i}(t,s)P_i(s) \right \|&\leq& K_i\exp(-\alpha_i(t-s)+\mu_i s),\quad t\geq s,\\
\left \| \Phi_{\lambda,i}(t,s)(I-P_i(s)) \right \|&\leq & K_i\exp(\alpha_i(t-s)+\mu_i s),\quad t\leq s. 
\end{array}
\right.
\end{equation*}

If we consider the evolution operator of system $\dot{x}=(U(t)-\lambda I)x$ and an invariant projection as follows
\begin{equation}
\label{diag}
\Phi_{\lambda}(t,s)=diag (\Phi_{\lambda,1}(t,s),\Phi_{\lambda,2}(t,s))\quad\textnormal{and}\quad P(t)=diag (P_1(t),P_2(t))
\end{equation}
and $K=\max \left \{ K_1,K_2\right \},\alpha=\min \left \{ \alpha_1,\alpha_2\right \}$ and $\mu=\max \left \{\mu_1,\mu_2\right \}$, then we have
\begin{equation}
\label{estimacion}
\left \{
\begin{array}{rcl}
%\label{nouniforme}
\left \| \Phi_{\lambda}(t,s)P(s) \right \|&\leq& K\exp(-\alpha(t-s)+\mu s),\quad t\geq s,\\
\left \| \Phi_{\lambda}(t,s)(I-P(s)) \right \|&\leq & K\exp(\alpha(t-s)+\mu s),\quad t\leq s,
\end{array}
\right.
\end{equation}
therefore $\lambda\in\rho(U)$. 

If $\lambda\in\rho(U)$, then the system $\dot{x}=(U(t)-\lambda I)x$ admits a nonuniform exponential dichotomy, which implies that it is satisfied (\ref{estimacion}). We can write $\Phi_{\lambda}(t,s)$ and $P(t)$ as in (\ref{diag}) and we can get the following estimates:
$$   
%\begin{equation*}
%\label{estimacion}
\begin{array}{rcl}
%\label{nouniforme}
\left \| \Phi_{\lambda,i}(t,s)P_i(s) \right \|\leq&\left \| \Phi_{\lambda}(t,s)P(s)\right \|&\leq K\exp(-\alpha(t-s)+\mu s),\\
\left \| \Phi_{\lambda,i}(t,s)(I-P_i(s)) \right \|\leq&\left \| \Phi_{\lambda}(t,s)(I-P(s))\right \|&\leq  K\exp(\alpha(t-s)+\mu s),
\end{array}
$$
%\end{equation*}
for $t\geq s$ and $t\leq s$ respectively,
which implies that $\lambda\in(\rho(U_1)\cap\rho(U_2))$.
}

\qed

\begin{remark}
\label{propiedadesproyector}
If $\lambda\notin\Sigma(A)$, it follows from Definition 5 that (\ref{sistemaperturbado}) has a nonuniform exponential dichotomy on $\R_0^{+}$ with projector $P_{\lambda}$. Nevertheless, it is interesting to note that:
\begin{itemize}
\item [a)] $Rank(P_{\lambda})$ is constant for any $\lambda\in(b_{i-1},a_i)$ ($i \in \lbrace 1, \dots , m \rbrace).$
\item [b)] If $\lambda_i \in(b_{i-1},a_i)$ and $\lambda_{i+1} \in (b_i,a_{i+1})$, then $Rank(P_{\lambda_i})<Rank(P_{\lambda_{i+1}}).$
\item [c)] $Rank(P_{\lambda})=0$ for any $\lambda\in(-\infty,a_i)$ and $Rank(P_{\lambda})=n$ for any
\newline $\lambda\in(b_m,+\infty).$
\end{itemize}
\end{remark}

\subsection{Main result.}

The main goal of this article is to prove the following result.
\begin{theorem}
If \textnormal{\textbf{(P1)}-\textbf{(P2)}} are satisfied, then (\ref{lin}) is nonuniformly contracted to $\Sigma(A)$.
\end{theorem}

In the next section, we will give the technical results which allow us to prove this Theorem.

\section{Preparatory results.}
The nonuniform kinematical similarity between (\ref{lin}) and (\ref{perturbado}) will be denoted by $A\cong U$. Let us recall that nonuniform kinematical similarity is an equivalence relation having several properties.
\begin{lemma}
If $A \cong B$, then $A-\lambda I \cong B-\lambda I$ for any $\lambda\in\R$.
\end{lemma}

\begin{proof}
If $A\cong B$ by the transformation $y(t)=S^{-1}(t)x(t)$, then $S(t)$ satisfies
$$B(t)=S^{-1}(t)A(t)S(t)-S^{-1}(t)\dot{S}(t).$$
It is straightforward see that
$$(B(t)-\lambda I)=S^{-1}(t)(A(t)-\lambda I)S(t)-S^{-1}(t)\dot{S}(t),$$
then $A(t)-\lambda I \cong B(t)-\lambda I.$
\end{proof}

\begin{lemma}
If $A \cong B$, then $\Sigma(A) = \Sigma(B)$.
\end{lemma}

\begin{proof}
Let $\lambda\in\rho(A)$ then the system $$\dot{x}=[A(t)-\lambda I]x $$ has a nonuniform exponential dichotomy on $\R_0^{+}$ with invariant projector P(t).

We define $\Phi_{B}(t,s)=S^{-1}(t)\Phi_{A}(t,s)S(t)$ which corresponds to the evolution operator associated to the system
 $$\dot{y}=[B(t)-\lambda I]y$$
and an invariant projector is $Q(t)=S^{-1}(t)P(t)S(t)$.

This fact combined with the submultiplicative property of norms and the estimates for $S$ and $S^{-1}$ allows to prove that if $t \geq s$ (the case $t\leq s$ can be proved similarly), then we deduce the estimation
\begin{equation*}
\begin{split}
\left\|\Phi_{B}(t,s)\exp(-\lambda(t-s))Q(s) \right \|
&\leq \left \| S^{-1}(t) \right \| \left \| \Phi_{A}(t,s)\exp(-\lambda(t-s))P(s) \right \| \left \| S(s) \right \| \\
&\leq M_{\upsilon}\exp(\upsilon t)K\exp(-\alpha(t-s)+\varepsilon s)M_{\upsilon}\exp(\upsilon s).
\end{split}
\end{equation*}

Finally, if $\alpha > \upsilon$, then $\lambda\in\rho(B)$.
In order to prove the other contention, we use the fact that $\cong$ is an equivalence relation.
\end{proof}

\begin{proposition}
%\begin{itemize}
%\item[(i)]
%If $\gamma \in \rho(A(t))$ y $J\subset\rho(A(t))$ is an interval containing $\gamma$, then the systems
%$\dot{x}=(A(t)-\gamma I)x$ and $\dot{x}=(A(t)-\eta I)x$ have a nonuniform exponential dichotomy with the same projector, for all $\eta\in J$.
%\item[(ii)]
If  $\Sigma(A)\subset [a,b]$ and $\lambda>b$ (resp. or $\lambda<a$) the system
$$\dot{x}=(A(t)-\lambda I)x$$
has a nonuniform exponential dichotomy with projector $P(t)=I$ (resp. with projector $P(t)=0$).
%\end{itemize}
\end{proposition}

\begin{proof}
    By condition \textbf{(P2)} we have that the evolution operator satisfies
    $$\left \| \Phi (t,s) \right \| \leq K_0\exp(L\left | t-s \right | + \bar{\varepsilon}\left | s \right |), \quad t,\ s \in\R_0^{+}.$$

     Since this is true, we consider $\lambda>b$. Let  $h=\max\left\lbrace L+1+\bar{\varepsilon},\lambda+1+\bar{\varepsilon} \right \rbrace $ and $$\Phi_h(t,s)=\Phi(t,s)\exp(-h(t-s)),$$
    then we obtain the following estimation
    $$\left \| \Phi_h(t,s) \right\|=\left \| \Phi(t,s) \right\|\exp(-h(t-s))\leq K_0 \exp(L(t-s)+\bar{\varepsilon}|s|-h(t-s)), \quad (t\geq s).$$

 Assuming this, now we define $\alpha=h-L>\bar{\varepsilon}$ and the previous equation becomes
    $$\left \| \Phi_h(t,s) \right\|\leq K_0 \exp(-\alpha(t-s)+\bar{\varepsilon}|s|) \quad (t\geq s),$$
    which implies that the system $$\dot{x}=(A(t)-hI)x$$ has a nonuniform exponential dichotomy with projector $P(t)=I$ and $[\lambda,h] \subset \rho(A).$

    By using Remark \ref{propiedadesproyector} the system $$\dot{x}=(A(t)-\lambda I)x$$ also has a nonuniform exponential dichotomy with projector $P(t)=I$.

    For $\lambda < a$ the proof is similar considering $h=\min\left\lbrace -(L+1+\bar{\varepsilon}),(\lambda-1-\bar{\varepsilon} \right )\rbrace.$
\end{proof}

The following result has been proved by X. Zhang  and J. Chu \textit{et al.}  by using the condition \textbf{(P2)}.

\begin{proposition}\textnormal{(\cite{Chu,Zhang})}
If the system (\ref{lin}) satisfies \textnormal{\textbf{(P1)--(P2)}} then its spectrum is as in (\ref{espectronouniforme}) and there exist m+2 matrix functions $B_i:\R\rightarrow M_{n_i}(\R)$ such that 
\begin{equation}
\label{ultimalabel}
\left \| B_i(t) \right \| \leq \mathcal{M}_i \exp(\mu_i t) \,\, \textnormal{with} \, \, \mu_i, \mathcal{M}_i > 0
\end{equation}
where $\Sigma(B_i)=[a_i,b_i]$ with $i \in \lbrace 1, \dots , m \rbrace$ and $\Sigma(B_0)=\Sigma(B_{m+1})=\emptyset$  such that (\ref{lin}) is nonuniformly kinematically similar to
\begin{equation}
\label{similaridadcinematicanouniforme}
\dot{y}=Diag(B_0(t), B_1(t),\dots, B_m(t), B_{m+1}(t))y.
\end{equation}
\end{proposition}

\begin{remark}
In our case the blocks $B_0(t)$ and $B_{m+1}(t)$ are omitted due to their dimensions $N_0$ and $N_{m+1}$ respectively are 0
\textnormal{(for more details, see \cite[Theorem 3.2]{S2002}).}

%In \cite{Chu,Zhang}) $B_0(t)$ and $B_{m+1}(t)$ appear on the diagonal of the matrix. If $(-\infty,b_1]$ and $[a_m,+\infty)$ are spectral intervals, in \cite[pp. 555]{Chu} blocks $B_0(t)$ and $B_{m+1}(t)$ are omitted.
\end{remark}

In \cite{Potzsche} it is introduced the concept of \textit{diagonal significance} which is fundamental for obtain the almost reducibility in the the case of exponential dichotomy \cite[Proposition 4]{CR2017} in a discrete context.

We point out that in \cite{Palmer2015} the concept of diagonal significance is studied in the continuous framework. In our case this condition it is not necessary. However, in the case of nonuniform exponential dichotomy the condition of diagonal significance is still open.

The following result is inspired by F. Lin \cite{Lin}, which mention a property about the spectrum of the upper triangular systems.

\begin{proposition}
Let $C(t)$ be an upper triangular $n\times n$-matrix function such that  $\Sigma(C)=[a,b]$, then
$$\bigcup_{i=1}^{n}\Sigma(c_{ii})\subset\Sigma(C),$$
where $c_{ii}(t)$ are the diagonal coefficients of $C(t).$
\end{proposition}

\begin{proof}
We will prove that $\bigcup_{i=1}^{n}\Sigma(c_{ii}) \subset \Sigma(C)$. Let $\lambda \notin \Sigma(C)=[a,b]$ such that $\lambda>b$. By Proposition 3, we have that the upper triangular system
\begin{equation}
\label{triangularperturbado}
\dot{x}=(C(t)-\lambda I)x,
\end{equation}
has a nonuniform exponential dichotomy with projector $P(t)=I$. That is, the evolution operator of (\ref{triangularperturbado}), namely $\Phi_{\lambda}(t,s)$, satisfies
$$\left \| \Phi_{\lambda}(t,s) \right\|\leq K_{\lambda} \exp(-\alpha_{\lambda}(t-s)+\varepsilon_{\lambda}|s|) \quad (t\geq s).$$

Now for each $i \in \lbrace 1, \dots , n \rbrace$, we have the following estimate
$$\exp\left( \int_{s}^{t} (c_{ii}(r)-\lambda) dr \right) \leq \left \| \Phi_{\lambda}(t,s) \right\| \leq K_{\lambda} \exp(-\alpha_{\lambda}(t-s)+\varepsilon_{\lambda}|s|)\quad (t\geq s),$$
and we conclude that the diagonal systems $$\dot{x_i}=(c_{ii}(t)-\lambda)x_i$$ have a nonuniform exponential dichotomy with projector $P(t)=1$ (scalar systems), which implies that $\lambda \notin \bigcup_{i=1}^{n}\Sigma(c_{ii})$.
\\

The case $\lambda<a$ can be proved analogously, thus $\bigcup_{i=1}^{n}\Sigma(c_{ii}) \subset \Sigma(C)$.
\end{proof}

\section{Proof of main results.}
\subsection{Proof of Theorem 1.} The proof will be made in several steps, whose ideas are inspired by the work developed by F. Lin in \cite{Lin}:

\textit{Step 1):} \textit{The system (\ref{lin}) is nonuniform kinematically similar to an upper triangular system:} By Proposition 1, there exists a positive integer $m\leq n$ such that:
$$\Sigma(A)=\bigcup_{i=1}^{m}[ a_i,b_i ],\quad \text{with $-\infty<a_1\leq b_1<\ldots<a_m\leq b_m <+\infty$.}$$
The Proposition 4 says that the systems (\ref{lin}) and (\ref{similaridadcinematicanouniforme}) are nonuniformly kinematically similar, where $B_i(t)$ are matrix function of order $n_i\times n_i$ satisfying (\ref{ultimalabel}) and
\newline $\Sigma(B_i)=[a_i,b_i]$ with $i \in \lbrace 1, \dots , m \rbrace$. Now, by using the method of QR factorization, we know that, for each $i \in \lbrace 1, \dots , m \rbrace$, the systems
\begin{equation}
\label{bloque_i}
\dot{x}_i=B_i(t)x_i
\end{equation}
are kinematically similar (see Definition 1 with $\upsilon=0$) to
\begin{equation}
\label{bloquetriangular_i}
\dot{y}_i=D_i(t)y_i,
\end{equation}
 where $D_i(t)$ is a  upper triangular $n_i\times n_i$-matrix function such that 
 $$\left \| D_i(t) \right \| \leq \mathcal{N}_i \exp( \varsigma_i t) \, \, \textnormal{and} \, \, \Sigma(D_i)=[a_i,b_i]$$

\vspace{0.3 cm}

\textit{Step 2):}  \textit{Nonuniform exponential dichotomy of scalar differential equation:} From now on, the diagonal terms of the upper triangular matrix $D_i$ described in (\ref{bloquetriangular_i}) will be denoted by $\left \{ d_{rr}^{(i)}\right \}_{r=1}^{n_i}$, where $i$ is a fixed element of $\left \{1,\dots,m\right \}$. Now, by Proposition 5, we have $$\bigcup_{r=1}^{n_i}\Sigma(d_{rr}^{(i)}) \subset \Sigma(D_i).$$

By Proposition 3, for any $\delta>0$ there exists $M_{\delta}=\frac{\delta}{m}>0$ such that the scalar differential equation

\begin{equation}
\label{escalarinestable}
\dot{x}=\left [ d_{rr}^{(i)}(t)-(a_i-M_{\delta}) \right ]x
\end{equation}
has a nonuniform exponential dichotomy on $\R_0^{+}$ with projector $P(t)=0$ and \begin{equation}
\label{escalarestable}
\dot{x}=\left [ d_{rr}^{(i)}(t)-(b_i+M_{\delta}) \right ]x
\end{equation}
has a nonuniform exponential dichotomy on $\R_0^{+}$ with projector $P(t)=1$. In consequence, there exist $\beta\geq 1$, $\alpha>0$, $\varepsilon\geq0$ and in this case we need the condition $\alpha>\varepsilon$,  such that 

\begin{equation}
\label{sistemasescalares}
\left\{
\begin{array}{rclr}
|\exp(\Phi(t,s))|&\leq& \beta  \exp(\alpha(t-s)+\varepsilon s) \quad &(t\leq s),\\
|\exp(\Psi(t,s))|&\leq& \beta  \exp(-\alpha(t-s)+\varepsilon s) \quad &(t\geq s),
\end{array}
\right.
\end{equation}

where
$$\exp(\Phi(t,s))=\exp\left(\int_{s}^{t}(d_{rr}^{(i)}(\tau)-(a_i-M_{\delta}))d\tau \right ),$$
and
$$\exp(\Psi(t,s))=\exp\left (\int_{s}^{t}(d_{rr}^{(i)}(\tau)-(b_i+M_{\delta}))d\tau \right ),$$
are the evolution operators of (\ref{escalarinestable}) and (\ref{escalarestable}) respectively.

\vspace{0.3 cm}

\textit{Step 3):} \textit{Upper and lower bounds for (\ref{sistemasescalares}):} For any fixed $i \in\left \{1,\dots,m\right \}$, there exist two functions $c_r^{(i)}$ and $\lambda_r^{(i)}$ such that
\begin{equation}
\label{funcionesatrozos}
a_i\leq c_r^{(i)}(t)\leq b_i  \quad\text{and}\quad |\lambda_r^{(i)}(t)|\leq M_{\delta} \quad \text{for any $t\in\R_0^{+}$}
\end{equation}
and there exist $\bar{\Delta}, \upsilon \geq0$ verifying
\begin{equation}
\label{estimacionintegral}
\left |\int_{0}^{t}[d_{rr}^{(i)}(\tau)-(c_r^{(i)}(\tau)+\lambda_r^{(i)}(\tau))]d\tau \right |\leq \bar{\Delta}+\upsilon t, \quad \text{if $t\geq0$}
\end{equation}
for any $r\in \left \{1,\dots,n_i\right \}$.

We will construct a strictly increasing and unbounded sequence of real numbers $\left \{ T_l^{(i)}\right \}_{l=0}^{+\infty}$ satisfying $T_0^{(i)}=0$ such that the function $c_r^{(i)}, \lambda_r^{(i)}:\R_0^{+}\rightarrow \R$ defined by:

$$c_r^{(i)}(t)=\left\{
\begin{array}{crl}
a_i &\quad\text{if} &\quad t \in [T_q^{(i)},T_{q+1}^{(i)}) \quad (q=0,2,4,\dots)\\
b_i &\quad\text{if}&\quad t \in [T_{q+1}^{(i)},T_{q+2}^{(i)})
\end{array}
\right.$$
%\end{itemize}
and
$$\lambda_r^{(i)}(t)=\left\{
\begin{array}{crl}
-M_{\delta} &\quad\text{if}&\quad t \in [T_q^{(i)},T_{q+1}^{(i)}) \quad (q=0,2,4,\dots)\\
M_{\delta}&\quad\text{if}&\quad t \in [T_{q+1}^{(i)},T_{q+2}^{(i)})
\end{array}
\right.$$
satisfy properties (\ref{funcionesatrozos}) and (\ref{estimacionintegral}) on $\R_0^{+}$.

It is straightforward to see that (\ref{funcionesatrozos}) is always satisfied. In order to verify (\ref{estimacionintegral}), we interchange $t$ by $s$ in the first inequality of (\ref{sistemasescalares}), then we have:

\begin{equation}
\label{estimaciones}
\left \{
\begin{array}{rclr}
\Phi(t,s) &\geq& \alpha(t-s)-\varepsilon t -\ln(\beta) \quad& (t\geq s),\\
\Psi(t,s) &\leq& -\alpha(t-s)+\varepsilon s+\ln(\beta) \quad& (t\geq s).
\end{array}
\right.
\end{equation}

By using induction, we will verify that there exists a sequence $\left \{ T_l^{(i)}\right \}_{l=0}^{+\infty}$ satisfying (\ref{estimacionintegral}).
In order to start, by the Proposition 2 combined with the fact that
$$\exp \left (\int_{s}^{t}d_{rr}^{(i)}(\tau)d\tau\right)\leq\left \| \Phi_{D_i}(t,s) \right \|,$$
where $\Phi_{D_i}(t,s)$ is the evolution operator of the system (\ref{bloquetriangular_i}),
then there exist cons\-tants $\bar{a}\geq0$, $\bar{\varepsilon}\geq0$ and $K\geq1$ such that the following inequality is satisfies
\begin{equation}
\label{porarriba}
\int_{s}^{t}d_{rr}^{(i)}(\tau)d\tau\leq \bar{a}|t-s|+\bar{\varepsilon}s+\ln(K).
\end{equation}

Then, using the equation (\ref{porarriba}) we have
\begin{equation}
\label{porarribaperturbado}
\Phi(t,s)\leq \bar{a}(t-s)+\bar{\varepsilon}s+\ln(K)+|a_i|(t-s)+M_{\delta}(t-s).
\end{equation}

On the other hand, by (\ref{estimaciones}) we obtain
\begin{equation}
%\label{porabajo}
\int_{s}^{t}d_{rr}^{(i)}(\tau)d\tau\geq\alpha(t-s)-\varepsilon t+a_i(t-s)-M_{\delta}(t-s)-\ln(\beta),
\end{equation}
and using the last expression we deduce
\begin{equation}
\label{porabajoperturbado}
\Psi(t,s)\geq (-\alpha-(b_i-a_i)-2M_{\delta})(t-s)-\varepsilon t-\ln(\beta).
\end{equation}

By the equations (\ref{estimaciones}), (\ref{porarribaperturbado}) and (\ref{porabajoperturbado}) we have the following

\begin{equation}
\label{estimacionarriba}
\left \{
\begin{array}{rclr}
\Phi(t,s) &\geq& (\alpha-\varepsilon)(t-s)-\varepsilon s -\ln(\beta) \quad &(t\geq s),\\
\Phi(t,s) &\leq& (\bar{a}+|a_i|+M_{\delta}+\bar{\varepsilon})(t-s)+ \bar{\varepsilon}s+\ln(K) \quad &(t\geq s).
\end{array}
\right.
\end{equation}
and 
%\vspace{0.5 cm}

\begin{equation}
\label{estimacionabajo}
\left\{
\begin{array}{rclr}
\Psi(t,s) &\leq& -(\alpha-\varepsilon)(t-s)+\varepsilon s+\ln(\beta) \quad & (t\geq s),\\
\Psi(t,s) &\geq& (-\varepsilon-(b_i-a_i)-2M_{\delta})(t-s) - \varepsilon s-\ln(\beta) \quad& (t\geq s).
\end{array}
\right.
\end{equation}

Now we will introduce constants and conditions  that allow us to obtain  the desired result (this conditions are inherent in the nonuniform framework).

Let $N, \xi, p, \bar{\xi}, \bar{p}\in \R$ constants that satisfy:
\begin{itemize}
\item[\textbf{(C1)}]$0<N<\min\left \{\alpha-\varepsilon, \bar{a}+|a_i|+M_{\delta}+\bar{\varepsilon}\right \}$.\\
\item [\textbf{(C2)}]$\max\left \{\ln(K), \ln(\beta)\right \}< p=-\bar{p}$.\\
%\item [\textbf{(C3)}]$\max\left \{-(\alpha-\varepsilon),-(\varepsilon+(b_i-a_i)+2M_{\delta})\right \}<\bar{N}\leq0$.
\item [\textbf{(C3)}]$0\leq\max\left \{\bar{\varepsilon}, \varepsilon\right \}\leq-\bar{\xi}\leq\xi$.
\end{itemize}

If $s=0$ in the first inequality of (\ref{estimacionarriba}) we obtain
$$\Phi(t,0) \geq (\alpha-\varepsilon)t-\ln(\beta), \quad t\geq0,$$
which implies that $\Phi(t,0)$ is unbounded in $\R_0^{+}$. In consequence, given $N, \xi, p \in\R$, there exists $T_1^{(i)}>0$ such that
\begin{equation}
\label{t1}
\left \{
\begin{array}{rclr}
\Phi(T_1^{(i)},0)&=&N(T_{1}^{(i)}-0)+\xi0+p,\\
\Phi(t,0)&<&N(t-0)+\xi0+p \quad &(0\leq t<T_{1}^{(i)}).
\end{array}
\right.
\end{equation}

Now we consider the value $\bar{\xi}T_{1}^{(i)}+\bar{p}$ and
$$T_{2}^{(i)}=\min\left \{\omega\in\R_0^{+}:\Psi(\omega,T_1^{(i)})=-N(T_{1}^{(i)}-0)-\xi0-p \right \},$$
with $T_2^{(i)}>T_1^{(i)}$. Consequently we will calculate the slope of the straight line that joins the points $\bar{\xi}T_{r_1}^{(i)}+\bar{p}$ and $-N(T_{r_1}^{(i)}-0)-\xi0-p$, which we will denote by $\bar{N}.$ Moreover,  $\bar{N}$ satisfies the following technical condition:
\medskip
\begin{itemize}
\item [\textbf{(C4)}]$\max\left \{-(\alpha-\varepsilon),-(\varepsilon+(b_i-a_i)+2M_{\delta})\right \}<\bar{N}$.
\end{itemize}
\medskip

Therefore we have that the calculation of slope of the line $\bar{N}$ is as follows
\begin{equation*}
\begin{split}
\bar{N}&=\frac{(-N(T_{1}^{(i)}-0)-\xi0-p)-(\bar{\xi}T_{1}^{(i)}+\bar{p})}{T_{2}^{(i)}-T_{1}^{(i)}},\\
\bar{N}&=\frac{-N(T_{1}^{(i)}-0)-\xi0)-\bar{\xi}T_{1}^{(i)}}{T_{2}^{(i)}-T_{1}^{(i)}}.
\end{split}
\end{equation*}

Due to the conditions \textbf{(C1)} and \textbf{(C3)}, we have that $\bar{N}\leq0$. In this way, we consider the straight line $\bar{N}(t-T_{1}^{(i)})+\bar{\xi}T_{1}+\bar{p}$.

 Based on the above and the equation (\ref{estimacionabajo}), if $s=T_1^{(i)}$, then there exists \newline $T_2^{(i)}>T_1^{(i)}$ such that
 \begin{equation}
\label{t2}
\left\{
\begin{array}{rclr}
\Psi(T_2^{(i)},T_1^{(i)})&=&\bar{N}(T_{2}^{(i)}-T_{1}^{(i)})+\bar{\xi}T_{1}+\bar{p},\\\\
\Psi(t,T_1^{(i)})&>&\bar{N}(t-T_{1}^{(i)})+\bar{\xi}T_{1}^{(i)}+\bar{p} \quad & (T_{1}^{(i)}\leq t<T_{2}^{(i)}).
\end{array}
\right.
\end{equation}

By (\ref{estimacionarriba}) and (\ref{t1}) we obtain that for $t\in[0,T_{1}^{(i)})$
$$-\varepsilon t-\ln(\beta)-Nt-\xi t-p\leq\int_{0}^{t}(d_{rr}^{(i)}(\tau)-(c_r^{(i)}(\tau)+\lambda_r^{(i)}(\tau)))d\tau\leq-\bar{N}t-\bar{\xi}t-\bar{p}.$$

In fact, we have
$$\varepsilon t-\ln(\beta)-(-\bar{N}t-\bar{\xi}t-\bar{p})\leq-\varepsilon t-\ln(\beta),$$
also by the equation (\ref{estimacionarriba}) we have
$$-\varepsilon t-\ln(\beta)\leq\int_{0}^{t}(d_{rr}^{(i)}(\tau)-(a_i-M_{\delta}))d\tau,$$ then by the equation (\ref{t1}) we can see that $$\int_{0}^{t}(d_{rr}^{(i)}(\tau)-(a_i-M_{\delta}))d\tau<N(t-0)+\xi0+p$$
and finally,
$$N(t-0)+\xi0+p\leq Nt+\xi t+p.$$

On the other hand, from the equations (\ref{estimacionabajo}) and (\ref{t2}), we have for $t\in[T_{1}^{(i)},T_{2}^{(i)})$

$$\varepsilon t+\ln(\beta)+Nt+\xi t+p\geq\int_{T_{1}^{(i)}}^{t}(d_{rr}^{(i)}(\tau)-(c_r^{(i)}(\tau)+\lambda_r^{(i)}(\tau)))d\tau\geq-(-\bar{N}t-\bar{\xi}t-\bar{p}).$$

Similarly to the previous estimations, we have that $$\varepsilon t+\ln(\beta)+Nt+\xi t+p\geq\varepsilon t+\ln(\beta),$$
by the equation (\ref{estimacionabajo}) we have
$$\varepsilon t+\ln(\beta)\geq\int_{T_{1}^{(i)}}^{t}(d_{rr}^{(i)}(\tau)-(b_i+M_{\delta}))d\tau,$$
and then by the equation (\ref{t2}) $$\int_{T_{1}^{(i)}}^{t}(d_{rr}^{(i)}(\tau)-(b_i+M_{\delta}))d\tau>\bar{N}(t-T_{1}^{(i)})+\bar{\xi}T_{1}^{(i)}+\bar{p}\geq-(-\bar{N}t-\bar{\xi}t-\bar{p}).$$

Thus for $t\in[0,T_{2}^{(i)})$
$$\left |\int_{0}^{t}(d_{rr}^{(i)}(\tau)-(c_r^{(i)}(\tau)+\lambda_r^{(i)}(\tau)))d\tau \right |\leq2\varepsilon t+2\ln(\beta)+2\max\left \{Nt+\xi t+p,-\bar{N}t-\bar{\xi}t-\bar{p} \right \}.$$

As inductive hypothesis, we will assume that there exists $2m+1$ numbers $$0=T_0^{(i)}<T_{1}^{(i)}<T_{2}^{(i)}< \dots<T_{{2m-1}}^{(i)}< T_{{2m}}^{(i)}$$
such that (\ref{funcionesatrozos}) is satisfies and for $t\in[0,T_{{2m}}^{(i)})$

$$\left |\int_{0}^{t}(d_{rr}^{(i)}(\tau)-(c_r^{(i)}(\tau)+\lambda_r^{(i)}(\tau)))d\tau \right |\leq2\varepsilon t+2\ln(\beta)+2\max\left \{Nt+\xi t+p,-\bar{N}t-\bar{\xi}t-\bar{p} \right \}.$$

By using the first inequality of (\ref{estimacionarriba}) and considering $s=T_{2m}^{(i)}$, we have that
$$\Phi(t,T_{2m}^{(i)}) \geq (\alpha-\varepsilon)(t-T_{2m}^{(i)})-\varepsilon T_{2m}^{(i)} -\ln(\beta)$$
is unbounded for any $t>T_{2m}^{(i)}$. Then, there exists $T_{{2m+1}}^{(i)}>T_{{2m}}^{(i)}$ such that

\begin{equation}
\label{t2m+1}
\left\{
\begin{array}{rcl}
\Phi(T_{2m+1}^{(i)},T_{2m}^{(i)})&=&N(T_{{2m+1}}^{(i)}-T_{{2m}}^{(i)})+\xi T_{{2m}}^{(i)}+p,\\\\
\Phi(t,T_{2m}^{(i)})&<&N(t-T_{{2m}}^{(i)})+\xi T_{{2m}}^{(i)}+p \quad (T_{{2m}}^{(i)}\leq t<T_{{2m+1}}^{(i)}).
\end{array}
\right.
\end{equation}

Now we consider the value $\bar{\xi}T_{{2m+1}}^{(i)}+\bar{p}$ and
$$T_{{2m+2}}^{(i)}=\min\left \{\omega\in\R_0^{+}:\Psi(\omega,T_{2m+1}^{(i)})=-N(T_{{2m+1}}^{(i)}-T_{{2m}}^{(i)})-\xi T_{{2m}}^{(i)}-p \right \},$$
with $T_{2m+2}^{(i)}>T_{2m+1}^{(i)}$. As before, let $\bar{N}$ be the slope of the straight line joining the points 
$$\bar{\xi}T_{{2m+1}}^{(i)}+\bar{p} \quad \textnormal{ and} \quad -N(T_{{2m+1}}^{(i)}-T_{{2m}}^{(i)})-\xi T_{{2m+1}}^{(i)}-p.$$ By the conditions \textbf{(C1)}, \textbf{(C2)} and \textbf{(C3)} we have $\bar{N}\leq0$. In this way, we consider the straight
line
$$\bar{N}(t-T_{{2m+1}}^{(i)})+\bar{\xi}T_{{2m+1}}^{(i)}+\bar{p}.$$

Combining the above straight line together with the equation (\ref{estimacionabajo}), we can see that if $s=T_{2m+1}^{(i)}$, then there exists $T_{2m+2}^{(i)}>T_{2m+1}^{(i)}$ such that
 \begin{equation}
\label{t2m+2}
\left\{
\begin{array}{rcl}
\Psi(T_{2m+2}^{(i)},T_{2m+1}^{(i)})&=&\bar{N}(T_{2m+2}^{(i)}-T_{2m+1}^{(i)})+\bar{\xi}T_{2m+1}+\bar{p},\\\\
\Psi(t,T_{2m+1}^{(i)})&>&\bar{N}(t-T_{2m+1}^{(i)})+\bar{\xi}T_{2m+1}^{(i)}+\bar{p},
\end{array}
\right.
\end{equation}
for $T_{2m+1}^{(i)}\leq t<T_{2m+2}^{(i)}.$

Now we will prove that for $t\in[0,T_{{2m+2}}^{(i)})$ we obtain
$$\left |\int_{0}^{t}(d_{rr}^{(i)}(\tau)-(c_r^{(i)}(\tau)+\lambda_r^{(i)}(\tau)))d\tau \right |\leq2\varepsilon t+2\ln(\beta)+2\max\left \{Nt+\xi t+p,-\bar{N}t-\bar{\xi}t-\bar{p} \right \}.$$

By inductive hypothesis, we have proved the case in which $t\in[0,T_{{2m}}^{(i)})$. If $t\in[T_{{2m}}^{(i)},T_{{2m+1}}^{(i)})$ we have
\begin{equation*}
\begin{split}
\left |\int_{0}^{t}(d_{rr}^{(i)}(\tau)-(c_r^{(i)}(\tau)+\lambda_r^{(i)}(\tau)))d\tau \right |
&=\left          |\int_{T_{{2m}}^{(i)}}^{t}(d_{         rr}^{(i)}(\tau)-(c_r^{(i)}(\tau)            +\lambda_r^{(i)}(\tau)))d\tau               \right |\\
&=\left                              |\int_{T_{{2m}}^{(i)}}^{t}(d_{    rr}^{(i)}(\tau)-(a_i-M_{\delta}))d\tau \right |.
\end{split}
\end{equation*}

By the equations (\ref{estimacionarriba}) and (\ref{t2m+1}),  as before we obtain that for $t\in[T_{{2m}}^{(i)},T_{{2m+1}}^{(i)})$
$$-\varepsilon t-\ln(\beta)-Nt-\xi t-p\leq\int_{T_{r_{2m}}^{(i)}}^{t}(d_{rr}^{(i)}(\tau)-(a_i-M_{\delta}))d\tau\leq-\bar{N}t-\bar{\xi}t-\bar{p}.$$

In the case $t\in[T_{{2m+1}}^{(i)},T_{{2m+2}}^{(i)})$, we have

$$
\begin{array}{rcl}
\left |\int_{0}^{t}(d_{rr}^{(i)}(\tau)-(c_r^{(i)}(\tau)+\lambda_r^{(i)}(\tau)))d\tau \right |& = &
\left |N(T_{{2m+1}}^{(i)}-T_{{2m}}^{(i)})+\xi T_{{2m}}^{(i)}+p \right.\\
&&+\left. \displaystyle\int_{T_{{2m+1}}^{(i)}}^{t}(d_{rr}^{(i)}(\tau)-(b_i+M_{\delta}))d\tau \right |.
\end{array}$$

Therefore by the equations (\ref{estimacionabajo}) and (\ref{t2m+2}), we have that for $t\in[T_{{2m+1}}^{(i)},T_{{2m+2}}^{(i)})$ it follows that

$$
\begin{array}{rcl}
\varepsilon t+\ln(\beta)+Nt+\xi t+p &\geq& N(T_{{2m+1}}^{(i)}-T_{{2m}}^{(i)})+\xi T_{{2m}}^{(i)}+p\\\\
&& +\displaystyle\int_{T_{{2m+1}}^{(i)}}^{t}(d_{rr}^{(i)}(\tau)-(b_i+M_{\delta}))d\tau\\\\
&& \geq-(-\bar{N}t-\bar{\xi}t-\bar{p}).
\end{array}
$$

In fact, (\ref{estimacionabajo}) and (\ref{t2m+2}) implies that

\begin{equation}
\label{tecnica}   \varepsilon t+\ln(\beta)\geq\int_{T_{{2m+1}}^{(i)}}^{t}(d_{rr}^{(i)}(\tau)-(b_i+M_{\delta}))d\tau\geq\bar{N}(t-T_{{2m+1}}^{(i)})+\bar{\xi} T_{{2m+1}}^{(i)}+\bar{p}.
\end{equation}

Now considering the two previous inequalities separately in (\ref{tecnica}), we obtain
$$
\varepsilon t+\ln(\beta)+N(T_{{2m+1}}^{(i)}-T_{{2m}}^{(i)})+\xi T_{{2m}}^{(i)}+p\geq \int_{0}^{t}(d_{rr}^{(i)}(\tau)-(b_i+M_{\delta}))d\tau
$$
and

$$
    \int_{0}^{t}(d_{rr}^{(i)}(\tau)-(b_i+M_{\delta}))d\tau\geq\bar{N}(t-T_{{2m+1}}^{(i)})+\bar{\xi} T_{{2m+1}}^{(i)}+\bar{p}+N(T_{{2m+1}}^{(i)}-T_{{2m}}^{(i)})+\xi T_{{2m}}^{(i)}+p.
$$

Then by the above first inequality, we have
$$\varepsilon t+\ln(\beta)+Nt+\xi t+p\geq\varepsilon t+\ln(\beta)+N(T_{{2m+1}}^{(i)}-T_{{2m}}^{(i)})+\xi T_{{2m}}^{(i)}+p$$
and by above second inequality, we obtain $$\bar{N}(t-T_{{2m+1}}^{(i)})+\bar{\xi} T_{{2m+1}}^{(i)}+\bar{p}+N(T_{{2m+1}}^{(i)}-T_{{2m}}^{(i)})+\xi T_{{2m}}^{(i)}+p\geq-(-\bar{N}t-\bar{\xi}t-\bar{p}).$$

Therefore, for $t\in[0,T_{{2m+2}}^{(i)})$ it follows that
$$\left |\int_{0}^{t}(d_{rr}^{(i)}(\tau)-(c_r^{(i)}(\tau)+\lambda_r^{(i)}(\tau)))d\tau \right |\leq2\varepsilon t+2\ln(\beta)+2\max\left \{Nt+\xi t+p,-\bar{N}t-\bar{\xi}t-\bar{p} \right \}.$$

Finally, we will prove that $T_{m}^{(i)}\rightarrow +\infty$ as $m\rightarrow +\infty$. For that, first of all we have by the equations (\ref{estimacionarriba}) and (\ref{t2m+1}):
\begin{equation*}
\begin{split}
N(T_{{2m+1}}^{(i)}-T_{{2m}}^{(i)})+\xi T_{{2m}}^{(i)}+p &=\int_{T_{{2m}}^{(i)}}^{T_{{2m+1}}^{(i)}}(d_{rr}^{(i)}(\tau)-(a_i-M_{\delta}))d\tau\\
&\leq(\bar{a}+|a_i|+M_{\delta}+\bar{\varepsilon})(T_{{2m+1}}^{(i)}-T_{{2m}}^{(i)})+ \bar{\varepsilon}T_{{2m}}^{(i)}+\ln(K),
\end{split}
\end{equation*}
which implies
$$(\xi-\bar{\varepsilon})T_{{2m}}^{(i)}+p-\ln(K)\leq(\bar{a}+|a_i|+M_{\delta}+\bar{\varepsilon}-N)(T_{{2m+1}}^{(i)}-T_{{2m}}^{(i)}).$$

By the conditions \textbf{(C1)} and \textbf{(C2)}, we have the following inequality

$$0<\frac{p-\ln(K)}{\bar{a}+|a_i|+M_{\delta}+\bar{\varepsilon}-N}\leq T_{{2m+1}}^{(i)}-T_{{2m}}^{(i)}.$$

On the other hand, in view of the equations (\ref{estimacionabajo}) and (\ref{t2m+2}) we have

\begin{equation*}
\begin{array}{rcl}
\bar{N}(T_{{2m+2}}^{(i)}-T_{{2m+1}}^{(i)})+\bar{\xi} T_{{2m+1}}^{(i)}+\bar{p}&=& \displaystyle \int_{T_{{2m+1}}^{(i)}}^{T_{{2m+2}}^{(i)}}(d_{rr}^{(i)}(\tau)-(b_i+M_{\delta}))d\tau\\
&\geq&(-\varepsilon-(b_i-a_i)-2M_{\delta})(T_{{2m+2}}^{(i)}-T_{{2m+1}}^{(i)})\\\\
&&-\varepsilon T_{{2m+1}}^{(i)}-\ln(\beta),
\end{array}
\end{equation*}
which implies that
$$(\bar{\xi}+\varepsilon)T_{{2m+1}}^{(i)}+\bar{p}+\ln(\beta)\geq-(\varepsilon+(b_i-a_i)+2M_{\delta}+\bar{N})(T_{{2m+2}}^{(i)}-T_{{2m+1}}^{(i)}).$$

Similarly, the conditions \textbf{(C2)} and \textbf{(C4)} allow us to ensure that

$$0<\frac{-\bar{p}-\ln(\beta)}{\varepsilon+(b_i-a_i)+2M_{\delta}+\bar{N}}\leq T_{{2m+2}}^{(i)}-T_{{2m+1}}^{(i)}.$$

Thus the above allows us to obtain the existence of $c_r^{(i)}(t)$, $\lambda_r^{(i)}(t)$ defined on $\R_0^{+}$ verifying (\ref{funcionesatrozos}) and finally:
$$\left |\int_{0}^{t}(d_{rr}^{(i)}(\tau)-(c_r^{(i)}(\tau)+\lambda_r^{(i)}(\tau)))d\tau \right |\leq\Delta+\upsilon t\quad (t\in\R_0^{+}),$$
with $\Delta\geq0$ and $\upsilon=\upsilon_{\varepsilon}$, defined by
%$\upsilon\geq0$ (where $\upsilon$ depends of $\varepsilon$).
\begin{equation}
\label{cotadependientedeepsilon}
\upsilon=\max\left \{2(\varepsilon+N+\xi),2(\varepsilon-\bar{N}-\bar{\xi})\right \}.
\end{equation}

Due to the definition of $c_r^{(i)}$ and $\lambda_r^{(i)},$ we know that they are piecewise continuous. Consequently, there exist continuous functions $\bar{c}_r^{(i)}$, $\bar{\lambda}_r^{(i)}$ satisfying

\begin{equation}
\label{funcionescontinuas}
a_i\leq \bar{c}_r^{(i)}(t)\leq b_i  \quad\text{and}\quad |\bar{\lambda}_r^{(i)}(t)|\leq M_{\delta} \quad \text{for any $t\in\R_0^{+}$}
\end{equation}
%$a_i\leq \bar{c}_r^{(i)}(t)\leq b_i, |\bar{\lambda}_r^{(i)}(t)|\leq M_{\delta}$
and
$$\int_{0}^{t}\left |(c_r^{(i)}(\tau)+\lambda_r^{(i)}(\tau))-(\bar{c}_r^{(i)}(\tau)+\bar{\lambda}_r^{(i)}(\tau))d\tau \right |\leq 1,$$
thus we deduce that
$$\left |\int_{0}^{t}[d_{rr}^{(i)}(\tau)-(\bar{c}_r^{(i)}(\tau)+\bar{\lambda}_r^{(i)}(\tau))]d\tau \right |\leq \bar{\Delta}+\upsilon t$$
with $\bar{\Delta}=\Delta+1$.

As a consequence of this result, we construct the $n_i\times n_i$ matrix:
$$L_i(t)=Diag(\mu_1(t),\dots,\mu_{n_i}(t)),$$
where for any $r\in \left \{1,\dots,n_i\right \}$, $\mu_r$ are defined by
$$\mu_r(t)=\exp\left (\int_{0}^{t}(d_{rr}^{(i)}(\tau)-(\bar{c}_r^{(i)}(\tau)+\bar{\lambda}_r^{(i)}(\tau))d\tau\right ),$$
and we conclude that $$\left \|L_i(t)\right \|\leq\Omega\exp(\upsilon t)\quad and \quad \left \|L_i^{-1}(t)\right \|\leq\Omega\exp(\upsilon t)\quad \text{for any}\quad t\in\R_0^{+},$$
with $\Omega=\exp(\bar{\Delta})$.

\vspace{0.3 cm}

 \textit{Step 4):} \textit{The systems (\ref{bloque_i}) can be nonuniformly contracted to $[a_i,b_i]$, for any $i\in\left \{1,\dots, m\right \}$:}
The system (\ref{bloquetriangular_i}) is nonuniformly kinematically similar to
\begin{equation}
\label{bloquelambda_i}
\dot{z}_i=\Lambda_i(t)z_i,
\end{equation}
with $y_i(t)=L_i(t)z_i(t)$, where $\Lambda_i(t)=L_i^{-1}(t)D_i(t)L_i(t)-L_i^{-1}(t)\dot{L}_i(t)$ is a $n_i\times n_i$ matrix whose rs-coefficient is defined by
\begin{equation*}
\left \{\Lambda_i(t)\right \}_{rs}=
\left\{
\begin{array}{rcl}
\bar{c}_r^{(i)}(t)+\bar{\lambda}_r^{(i)}(t)\quad& \text{if}\quad &r=s,\\
\frac{\mu_s(t)}{\mu_r(t)}d_{rs}^{(i)}(t) \quad& \text{if}\quad& 1\leq r<s\leq n_i,\\
0 \quad& \text{if}\quad& 1\leq s<r\leq n_i.
\end{array}
\right.
\end{equation*}
%Por medio del cambio de variable $y=L_i(t)z$, podemos ver que (\ref{bloquetriangular_i}) es no uniformemente cinem\'aticamente similar a

%\begin{equation}
%\label{diagonalperturbado}
%\dot{z}_i=\begin{pmatrix}
%\bar{c}_1^{(i)}(t)+\bar{\lambda}_1^{(i)}(t) &\exp(\int_{0}^{t}g_2^{(i)}(\tau)-g_1^{(i)}(\tau)d\tau)d_{12}^{(i)}(t)  &\cdots   &\exp(\int_{0}^{t}g_{n_{i}}^{(i)}(\tau)-g_1^{(i)}(\tau)d\tau)d_{1n_i}^{(i)}(t) \\
 %&\bar{c}_2^{(i)}(t)+\bar{\lambda}_2^{(i)}(t)  &\cdots  %&\exp(\int_{0}^{t}g_{n_{i}}^{(i)}(\tau)-g_2^{(i)}(\tau)d\tau)d_{2n_i}^{(i)}(t) \\
 %&  &\ddots   &\vdots  \\
 %&  &  &\bar{c}_{n_i}^{(i)}(t)+\bar{\lambda}_{n_i}^{(i)}(t)
%\end{pmatrix} z_i
%\end{equation}
%donde $g_l^{(i)}(t)=\exp\left (\int_{0}^{t}(d_{ll}^{(i)}(\tau)-(\bar{c}_l^{(i)}(\tau)+\bar{\lambda}_l^{(i)}(\tau))d\tau\right )$, para $l \in \lbrace 1, \dots , n_i \rbrace$ y adem\'as (\ref{bloque_i}) es no uniformemente cinem\'aticamente similar a (\ref{diagonalperturbado}).

We observe that $|d_{rs}^{(i)}(t)|\leq \mathcal{K}_1\exp(\kappa_1 t)$ with $\mathcal{K}_1>0$, for $1\leq r<s\leq n_i$ and by the equation (\ref{cotadependientedeepsilon}), we have $\frac{\mu_s(t)}{\mu_r(t)}\leq \mathcal{K}_2\exp(\kappa_2t)$, with $\mathcal{K}_2>1$ and $\kappa_2=2\upsilon$ with $\upsilon$ as in (\ref{cotadependientedeepsilon}), then
\begin{equation}
\label{estimaciongamma_i}
\left |\left \{\Lambda_{i}(t)\right \}_{rs}\right | \leq \mathcal{K}_2 \mathcal{K}_1\exp(\kappa t),
\end{equation}
where $\kappa=\kappa_{\varepsilon}=\kappa_1+\kappa_2$.

Let us define the transformation
$$z_i(t)=R_i(t)w_i(t),$$
with
%\begin{equation*}
%\begin{split}
$$R_i(t)=Diag(\exp(-\kappa M_{\delta}K_{\delta}t),\eta\exp(-2\kappa M_{\delta}K_{\delta}t),\dots,
\eta^{n_i-1}\exp(-n_i \kappa M_{\delta}K_{\delta}t)),$$
%\end{split}
%\end{equation*}
and we also define $K_{\delta}$ such that $K_{\delta}M_{\delta}\geq1$ and
\begin{equation}
\label{eta}
0<\eta<\frac{M_{\delta}}{M_{\delta}+ \mathcal{K}_1 \mathcal{K}_2}. %donde $\delta$ es suficientemente peque\~no tal que $\eta<1$
\end{equation}
%\begin{equation*}
%\begin{split}
%R_i(t)=Diag(\exp(-AM_{\delta}K_{\delta}t),\eta\exp(-2AM_{\delta}K_{\delta}t),\dots,
%\eta^{n_i-1}\exp(-n_iAM_{\delta}K_{\delta}t))
%\end{split}
%\end{equation*}
%y consideremos el cambio de variables $z_i=R_iw_i$, entonces obtenemos lo siguiente

Now, we can see that (\ref{bloquelambda_i}) and (\ref{bloquetriangular_i}) are $\delta$-nonuniformly kinematically similar to

$$\dot{w}_i=\Gamma_i(t)w_i,$$
where the $rs$-coefficient of $\Gamma_i(t)$ is

\begin{equation*}
\left \{\Gamma_i(t)\right \}_{rs}=
\left\{
\begin{array}{rcl}
\left \{\Lambda_i(t)\right \}_{rs}+r \kappa M_{\delta}K_{\delta}\quad& \text{if}\quad& r=s,\\
\eta^{s-r}\left \{\Lambda_i(t)\right \}_{rs}
\exp(-(s-r) \kappa M_{\delta}K_{\delta}t)\quad& \text{if}\quad& 1\leq r<s\leq n_i,\\
0 \quad& \text{if}\quad& 1\leq s<r\leq n_i.
\end{array}
\right.
\end{equation*}

Let us observe that $\Gamma_i(t)$ can be written as follows:

$$\Gamma_i(t)=\bar{C}_i(t)I+\bar{B}_i(t),$$

where $\bar{C}_i(t)=Diag(\bar{c}_1(t),\dots,\bar{c}_{n_i}(t))$ and the $rs$-coefficient of $\bar{B}_i(t)$ is defined by
\begin{equation*}
\left \{\bar{B}_i(t)\right \}_{rs}=
\left\{
\begin{array}{rcl}
\left \{\bar{\lambda}_r(t)\right \}_{rs}+r \kappa M_{\delta}K_{\delta}\quad& \text{if}\quad& r=s,\\
\eta^{s-r}\frac{\mu_s(t)}{\mu_r(t)}d_{rs}^{(i)}(t)
\exp(-(s-r) \kappa M_{\delta}K_{\delta}t)\quad& \text{if}\quad& 1\leq r<s\leq n_i,\\
0 \quad& \text{if}\quad& 1\leq s<r\leq n_i.
\end{array}
\right.
\end{equation*}

%\begin{equation*}
%\dot{w}_i=
%\begin{pmatrix}
%\bar{c}_1^{(i)}(t)+\bar{\lambda}_1^{(i)}(t)+AM_{\delta}K_{\delta} &\eta\exp(-AM_{\delta}K_{\delta}t)h_{12}^{(i)}(t) &\cdots   &\eta^{n_i-1}\exp(-(n_i-1)AM_{\delta}K_{\delta}t)h_{1n_i}^{(i)}(t) \\
 %&\bar{c}_2^{(i)}(t)+\bar{\lambda}_2^{(i)}(t)+2AM_{\delta}K_{\delta}  &\cdots  &\eta^{n_i-2}\exp(-(n_i-2)AM_{\delta}K_{\delta}t)h_{2n_i}^{(i)}(t) \\
 %&  &\ddots   &\vdots  \\
 %&  &  &\bar{c}_{n_i}^{(i)}(t)+\bar{\lambda}_{n_i}^{(i)}(t)+n_iAM_{\delta}K_{\delta}
%\end{pmatrix}w_i
%\end{equation*}

%A continuaci\'on definimos
%\begin{equation*}
%\bar{B}_i(t)=
%\begin{pmatrix}
%\bar{\lambda}_1^{(i)}(t)+AM_{\delta}K_{\delta} &\eta\exp(-AM_{\delta}K_{\delta}t)h_{12}^{(i)}(t) &\cdots   &\eta^{n_i-1}\exp(-(n_i-1)AM_{\delta}K_{\delta}t)h_{1n_i}^{(i)}(t) \\
 %&\bar{\lambda}_2^{(i)}(t)+2AM_{\delta}K_{\delta}  &\cdots  &\eta^{n_i-2}\exp(-(n_i-2)AM_{\delta}K_{\delta}t)h_{2n_i}^{(i)}(t) \\
 %&  &\ddots   &\vdots  \\
 %&  &  &\bar{\lambda}_{n_i}^{(i)}(t)+n_iAM_{\delta}K_{\delta}
%\end{pmatrix}
%\end{equation*}

By (\ref{funcionescontinuas}) and (\ref{estimaciongamma_i}), we can verify that

\begin{equation*}
\left \|\bar{B}_i(t)\right \|\leq M_{\delta}[1+r \kappa K_{\delta}]+ \mathcal{K}_1 \mathcal{K}_2[\eta+\eta^2+\cdots+\eta^{n_i}]\leq M_{\delta}[1+n_i \kappa K_{\delta}]+ \mathcal{K}_1 \mathcal{K}_2\frac{\eta}{1-\eta}.
\end{equation*}

Recall that $M_{\delta}=\frac{\delta}{m}$ and by using (\ref{eta}) it follows that $\left \|\bar{B}_i(t)\right \|\leq\frac{\delta}{m}K_{\delta,\varepsilon}$, where $K_{\delta,\varepsilon}=2+n_i \kappa K_{\delta}$.

Thus, for any $i \in \lbrace 1, \dots , m \rbrace$ the system  (\ref{bloquetriangular_i}) is $\bar{\delta}$-nonuniformly kinematically similar (with $\bar{\delta}=\frac{\delta}{m}$) to

%y notemos que $\left \|\bar{B}_i(t)\right \|\leq \frac{\delta}{m}K_{\delta,\varepsilon}$, con $K_{\delta,\varepsilon}>1$, as\'i que (\ref{diagonalperturbado}) es no uniformemente cinem\'aticamente similar a
\begin{equation*}
\label{diagonal+perturbado_i}
\dot{w}_i=[\bar{C}_i(t)+\bar{B}_i(t)]w_i,
\end{equation*}
where
$$\bar{c}_j(t)\in[a_i,b_i]=\Sigma(B_i), \quad j \in \lbrace 1, \dots , n_i \rbrace \quad\text{and}\quad  \left \|\bar{B}_i(t)\right \|\leq\frac{\delta}{m}K_{\delta,\varepsilon}.$$

Finally, (\ref{bloquetriangular_i}) is nonuniformly contracted to $\Sigma(B_i)$.

\vspace{0.3 cm}

\textit{Step 5):} \textit{The system (\ref{lin}) can be nonuniformly contracted to $\Sigma(A)$:} By using the previous result, we can see that (\ref{lin}) is $\delta$-nonuniformly kinematically similar to

\begin{equation*}
\label{diagonal+perturbado}
\dot{w}=[C(t)+ B(t)]w,
\end{equation*}
where
$$C(t)=Diag(\bar{C}_1(t),\dots, \bar{C}_m(t))\quad\text{and}\quad B(t)=Diag(\bar{B}_1(t),\dots,\bar{B}_m(t)).$$

In consequence, note that

$$C(t)\subset\bigcup_{i=1}^{m}[ a_i,b_i ]=\Sigma(A)\quad\text{and}\quad\left \| B(t)\right \|\leq\delta K_{\delta,\varepsilon}.$$

Finally, the system (\ref{lin}) is nonuniformly contracted to $\Sigma(A)$.

\qed

%\begin{remark}
%Explicar las condiciones C1--C4.
%De las ecuaciones \textnormal{(\ref{estimacionarriba})}
%y \textnormal{(\ref{estimacionabajo})} se puede ver que las funciones $\Phi(t,s)$ y %$\Psi(t,s)$ est\'an acotadas superior e inferiormente por rectas, las cuales no necesariamente intersectan a las funciones $\Phi(t,s)$ y $\Psi(t,s)$, por esto las condiciones \textnormal{\bf(C1)--(C4)} nos permiten encontrar rectas que intersectan a estas funciones al menos una vez y una sucesi\'on de puntos $\left \{ T_l^{(i)}\right \}_{l=0}^{+\infty}$ los cuales describen la primera vez que se intersectan las rectas con las funciones. Esta sucesi\'on de puntos no se acumulan, es decir, tiende a infinito.
%\end{remark}

\begin{remark}
The inequalities \textnormal{(\ref{estimacionarriba})}
and \textnormal{(\ref{estimacionabajo})} show that the functions $\Phi(t,s)$ and $\Psi(t,s)$ are bounded by specific functions. This functions not necessarily cross to graph of functions $\Phi(t,s)$ and $\Psi(t,s)$, thus the conditions  \textnormal{\bf(C1)--(C4)} allows us to find straight lines that cross it at least once. Moreover this procedure enable us to construct  $\left \{ T_l^{(i)}\right \}_{l=0}^{+\infty}$, which is the sequence of first crossing times of the graph of the functions $\Phi(t,s)$ and $\Psi(t,s)$ with those straight lines. We have proved that this sequence of crossing time has not accumulations points.
\end{remark}

\section{Application of the main result}

In this section, we present two examples of scalar systems, with their respective spectra and, in addition, we will prove that there exists a Lyapunov transformation that allows each system is contracted to its spectrum.  Finally, we will present a diagonal planar example considering the previous scalar systems.
\begin{example} 
By one hand, let us consider the scalar differential equation studied in {\cite[p.547]{Chu}}\textnormal{:}
\begin{equation}
\label{ejemplo}
    \begin{array}{lccccc}
    \dot{x}=A_1(t)x,&\textnormal{with}&A_1(t)=\lambda_0 + at\sin(t)&\textnormal{and}&\lambda_0 <a<0.  
    \end{array}
\end{equation}

It is straightforward to verify that the example can be adapted to the case $\R_0^{+}$ and therefore the spectrum of \textnormal{(\ref{ejemplo})} is $\Sigma(A_1)=[\lambda_0+a,\lambda_0-a].$

We claim that \textnormal{(\ref{ejemplo})} is nonuniformly contracted to $\Sigma(A)$. Indeed, given a fixed $\delta>0$ and $\varepsilon_1=2|a|$, we consider the matrix function $t\rightarrow S_1(t)\in M_1(\R)$ defined by
$$S_1(t)=\exp \left(\frac{\varepsilon_1}{2} t\cos(t)-\delta\sin(t)\right),$$
and we can verify that \textnormal{(\ref{ejemplo})} is $\delta$-nonuniformly kinematically similar to
\begin{equation*} 
\begin{array}{ccccc}
\dot{y}=(C_1(t)+B_1(t))y, &\textnormal{with}&C_1(t)=\lambda_0&\textnormal{and}&B_1(t)=-\delta\cos(t)\left(1+\frac{\varepsilon_1}{2\delta}\right).
\end{array}
\end{equation*}

The result is followed as $C_1(t)\in [\lambda_0+a,\lambda_0-a]$ and $\left \|B_1(t) \right \|\leq \delta K_{\delta,\varepsilon_1}$, where $K_{\delta,\varepsilon}=1+\frac{\varepsilon_1}{2\delta}$.

On the other hand, let us consider the scalar differential equation shown in \cite{Zhu}\textnormal{:}
\begin{equation}
\label{ejemplo2}
    \dot{x}=A_2(t)x,\quad \textnormal{where} \quad A_2(t)=\lambda_1(\sin(\ln(t+1))+\cos(\ln(t+1)))
    \end{equation}
with $\lambda_1\neq0.$
After calculations we can see that the spectrum of \textnormal{(\ref{ejemplo2})} is $$\Sigma(A_2)=[-|\lambda_1|,|\lambda_1|].$$

Now we show that \textnormal{(\ref{ejemplo2})} is nonuniformly contracted to $\Sigma(A_2)$. Indeed, given a fixed $\delta>0$ and $\varepsilon_2=2\lambda_1$, we consider the matrix function $t\rightarrow S_2(t)\in M_1(\R)$ defined by $$S_2(t)=\exp\left(\frac{1}{2}\delta[(t+1)\sin(\ln(t+1))+(t+1)\cos(\ln(t+1))]\right ),$$ and we can verify that \textnormal{(\ref{ejemplo2})} is $\delta$-nonuniformly kinematically similar to
\begin{equation*} 
%\begin{array}{ccccc}
\dot{y}=(C_2(t)+B_2(t))y, 
\end{equation*}
with $C_2(t)=\lambda_1\sin(\ln(t+1))$ and $B_2(t)=\lambda_1\cos(\ln(t+1))+\delta\cos(\ln(t+1))$.
%\end{array}
%\end{equation*}

Finally, since $C_2(t)\in [-|\lambda_1|,|\lambda_1|]$ and $\left \|B_2(t) \right \|\leq \delta K_{\delta,\varepsilon_2}$, where $K_{\delta,\varepsilon_2}=1+\frac{\varepsilon_2}{2\delta}$; we obtain that \textnormal{(\ref{ejemplo2})} is contracted to 
$\Sigma(A_2)$.
\end{example}

\begin{example}
Consider the planar system defined on $\R_0^{+}$ as follows
\begin{equation}
\label{ejemploplanar}
\dot{x}=A(t)x,
\end{equation}
where 
$$A(t)=\begin{pmatrix}
\lambda_0 + at\sin t & 0\\ 
0 & \lambda_1(\sin(\ln(t+1))+\cos(\ln(t+1)))
\end{pmatrix}$$
with the same conditions as in the example 1 and also $\lambda_0-a<-|\lambda_1|$. We can see that in view of the lemma \ref{lemaespectrodiagonal}, the nonuniform spectrum of \textnormal{(\ref{ejemploplanar})} is $$\Sigma(A)=[\lambda_0+a,\lambda_0-a]\cup[-|\lambda_1|,|\lambda_1|].$$

Now we claim that (\ref{ejemploplanar}) is nonuniformly contracted to $\Sigma(A)$. Indeed, given a fixed $\delta>0$ and $\varepsilon=\max\left \{\varepsilon_1,\varepsilon_2\right \}$, we consider the matrix function $t\rightarrow S(t)\in M_2(\R)$ defined by 
$$S(t)=\begin{pmatrix}
S_1(t)&0\\
0&S_2(t)
\end{pmatrix},$$
where $S_1(t)$ and $S_2(t)$ are as in the previous examples. It is verified that \textnormal{(\ref{ejemploplanar})} is $\delta$-nonuniformly kinematically similar to 
\begin{equation*} 
%\begin{array}{ccccc}
\dot{y}=(C(t)+B(t))y, 
\end{equation*}
with 
$$C(t)=\begin{pmatrix}
\lambda_0 & 0\\
0 & \lambda_1\sin(\ln(t+1))
\end{pmatrix}$$
and 
$$B(t)=\begin{pmatrix}
-\delta\cos(t)\left (1+\frac{\varepsilon_1}{2\delta} \right ) & 0 \\
0 & \lambda_1\cos(\ln(t+1))+\delta\cos(\ln(t+1))
\end{pmatrix}.$$

The claim follows since $C(t)\in\Sigma(A)$ and $\left \| B(t)\right \|\leq\delta K_{\delta,\varepsilon}$, where

$$K_{\delta,\varepsilon}=\max\left \{K_{\delta,\varepsilon_1},K_{\delta,\varepsilon_2}\right \}.$$

\end{example}

\section{Conclusions}
This paper can be seen as a progress report on nonuniform hyperbolicity and its associated spectrum.  In particular, we have introduced the concepts of nonuniform almost reducibility and nonuniform contractibility in order to proved that if a linear nonautonomous differential system verifies a subtle hypothesis of nonuniform hyperbolicity on the half line, then this system is nonuniformly contracted to the spectrum of nonuniform exponential dichotomy.  Additionally, it is an advance in the subject of nonautonomous differential equations in terms of the obtain a simpler form for a linear system with nonuniform hyperbolicity.

We point out that in \cite{Huerta} it has used strongly the main result of this article for prove the topological equivalence in the half line between a linear system with nonuniform contraction and an unbounded nonlinear perturbation.

\end{document}